\documentclass[12 pt]{amsart}

\usepackage{layout}
\usepackage{amsmath, amsthm}
\usepackage{amssymb, amsrefs}
\usepackage{stmaryrd}
\usepackage{graphicx}
\usepackage{setspace}
\usepackage{textcomp}
\usepackage{esvect}
\usepackage[letterpaper]{geometry}
\usepackage{arydshln}
\usepackage[all]{xy}
\usepackage{etex}
\usepackage{ifthen}

\usepackage[dvipsnames]{xcolor}

\usepackage{epsfig}
\usepackage{latexsym}
\usepackage{amsthm}

\usepackage{mathrsfs}

\usepackage[colorlinks,citecolor=blue]{hyperref}

\usepackage[latin1]{inputenc}

\usepackage{tikz-cd}
\usetikzlibrary{graphs,graphs.standard}
\usepackage{pgfplots}
\usetikzlibrary{matrix,arrows,decorations.pathmorphing}
\usetikzlibrary{shapes.geometric}

\newtheorem{theorem}{Theorem}[section]
\newtheorem{lemma}[theorem]{Lemma}
\newtheorem{proposition}[theorem]{Proposition}

\newtheorem{corollary}[theorem]{Corollary}

\theoremstyle{definition}
\newtheorem{definition}[theorem]{Definition}

\newtheorem{example}[theorem]{Example}

\newtheorem{obs}[theorem]{Observation}

\theoremstyle{definition}
\numberwithin{equation}{section}

\tikzset{
    ncbar angle/.initial=90,
    ncbar/.style={
        to path=(\tikztostart)
        -- ($(\tikztostart)!#1!\pgfkeysvalueof{/tikz/ncbar angle}:(\tikztotarget)$)
        -- ($(\tikztotarget)!($(\tikztostart)!#1!\pgfkeysvalueof{/tikz/ncbar angle}:(\tikztotarget)$)!\pgfkeysvalueof{/tikz/ncbar angle}:(\tikztostart)$)
        -- (\tikztotarget)
    },
    ncbar/.default=0.5cm,
}

\tikzset{square left brace/.style={ncbar=0.5cm}}
\tikzset{square right brace/.style={ncbar=-0.5cm}}

\tikzset{round left paren/.style={ncbar=0.5cm,out=120,in=-120}}
\tikzset{round right paren/.style={ncbar=0.5cm,out=60,in=-60}}

\DeclareMathOperator{\id}{id}

\def \Gph {\mathsf{Gph}}

\def \con {\sim}
\definecolor{laura}{rgb}{.4, 0, .6}

\begin{document}
\pagestyle{headings}

 \title{Cops and Robbers: A $\times$-homotopy invariant variant}
    
    \author{T. Chih and L. Scull}

\vspace*{-2em}

\begin{abstract}
Cops and Robbers is a pursuit-evasion game played on graphs, of which many variants have been developed and studied.  We introduce a variant of this game, "Sneaky-Active Cops and Robbers", where all cops and robber must move on their turn, and where the robber is allowed to move onto a cop position without being captured.  We show that for reflexive graphs, this game is equivalent to the classical cops and robbers and  that the cop number for a graph is invariant under $\times$-homotopy equivalence.  We then develop further properties of this game, computing cop numbers for a number of graph families and developing results about the behavior of categorical and box products of graphs.  

\end{abstract}

\maketitle
\setcounter{tocdepth}{1}
\tableofcontents
\section{Introduction}

The game of cops and robbers is a pursuit-evasion game played on discrete graphs, introduced independently by Quillot \cite{Quillot} and Nowakowski and Winkler \cite{Nowakowski}.  The game has two players, one controlling a group of cops and the other controlling a robber.   The players take turns moving the pieces under their control, either moving to an adjacent vertex or staying in the same place.    The cops win by having a cop occupy the same vertex as the robber, and the robber wins by indefinitely evading the cops.  Numerous variants have been developed and studied.  In particular, Aigner and Fromme \cite{CROriginal} and  Neufeld and Nowakowski \cite{CRProduct} considered an ``active" variant where each player must move at least one piece to an adjacent vertex on each turn, and   Gromovikov, Kinnersley and Simone \cite{FullyActive} considered a ``fully active" modification where all pieces must move on every turn.

Independent from this, the development of discrete homotopy theories defined on graphs have been an ongoing area of inquiry.    There are two prominent theories of homotopies for graphs in the literature:     $A$-homotopy \cites{Babson1, Barcelo1, BoxHomotopy,  hardeman2019lifting} and $\times$-homotopy \cites{CS1, CS2, CS3, Docht1, Docht2, CompGH, Kozlov1, Kozlov1, Kozlov3, KozlovShort}.  We will focus on   $\times$-homotopy, in which  two finite graphs $X, Y$ are homotopy equivalent when we may transform one to another via a sequence of  ``fold" moves.  
For reflexive graphs, this folding process is equivalent to the  dismantling operation studied in in the Cops and Robber literature \cite{BonatoCaR}.  It is known that dismantling preserves the cop number of a graph.   Inspired by this, we construct a variant of Cops and Robbers which  is $\times$-homotopy invariant for all graphs.  The active and fully active variants are not  $\times$-homotopy invariant, so we  define a related  ``sneaky-active" version of the game, where all pieces must move on their respective turns, and the robber is allowed to move onto a cop position without being captured, so that capture only occurs when the cop moves onto the robber.   The goal of this paper is to study this new sneaky-active variant of Cops and Robbers.  In particular, we will show that this variant generalizes the classic game and agrees with it on reflexive graphs (Proposition \ref{P:reflexiveclassic}), and is  $\times$-homotopy equivalent  for any graph (Theorem \ref{T:invariant}).

The paper is organized as follows.  In  Section \ref{S:background}, we introduce the requisite background for both$\times$-homotopy and  Cops and Robbers.  In Section \ref{S:basics}, we define  the ``Sneaky-Active Cops and Robbers" game variant,  show that this generalizes the classic game and is $\times$-homtopy invariant, and give some comparisons to previously studied variants.  In Section \ref{S:families}, we compute the sneaky-active cop numbers for a variety of graph families.  In Section \ref{S:CP}, we show that the sneaky-active cop number for the categorical product of graphs is determined completely by the cop number of the factor graphs.  Finally in Section \ref{S:BP}, we find   bounds for the sneaky-active cop number of  box products of graphs, both generally and in some specific graph families.

\section{Background} \label{S:background}

In this section, we summarize background material regarding $\times$-homotopy and Cops and Robbers.  More details on $\times$-homotopy be found in \cite{CS1}, and more on Cops and Robbers may be found in \cites{BonatoCaR, FullyActive}.      We work in the category $\Gph$ of  undirected graphs, without multiple edges and where loops are allowed.  Moreover, throughout this paper, we will assume that graphs are  finite and contain no isolated vertices.  Graph theory terminology and notation follows \cite{Bondy} and category theory terminology and notation follows \cite{riehlCTIC}.

\subsection{Background on $\times$-Homomotopy}
\begin{definition} \cite{HN2004} The category of finite graphs $\Gph$ is defined by: \begin{itemize} \item An object is a  graph $X$,  consisting of a finite set of vertices $V(X)$ and a set $E(X)$ of edges connecting them.  Each edge is given by an unordered pair of vertices.     Any pair of vertices has at most one edge connecting them, and loops are allowed but isolated vertices are not:  each vertex must be connected to at least one other (possibly itself).     A connecting  edge will be notated by   $v_1 \con v_2$. 
\item 
A morphism in the category   $\Gph$ is a graph homomorphism $f:  X \to Y$, given by a set map $f:  V(X) \to V(Y)$ such that for $v_1, v_2 \in V(X)$, if $v_1 \con v_2 \in E(X)$ then $f(v_1) \con f(v_2) \in E(Y)$.  \end{itemize} \end{definition}  

Throughout this paper, we will assume that  `graph' always refers to an object in $\Gph$.   

To define $\times$-homotopy, we use the looped path graph. 

\begin{definition}\label{D:path} \cites{Bondy, Docht1}   Let $I_n^{\ell}$ be the looped path graph with  $n+1$ vertices   $\{ 0, 1, \dots, n\}$ such that $i \con i$ and  $i \con {i+1}$ for $i=0, \ldots, n-1$.  

$$ 
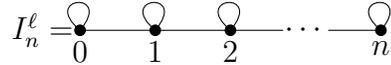
\begin{tikzpicture}
\node at (-.5,0){$I_n^{\ell} = $};
\draw[fill] (0,0) circle (2pt);
\draw (0,0) --node[below]{0} (0,0);
\draw (0,0)  to[in=50,out=140,loop, distance=.7cm] (0,0);

\draw[fill] (1,0) circle (2pt);
\draw (1,0) --node[below]{1} (1,0);
\draw (1,0)  to[in=50,out=140,loop, distance=.7cm] (1,0);

\draw[fill] (2,0) circle (2pt);
\draw (2,0) --node[below]{2} (2,0);
\draw (2,0)  to[in=50,out=140,loop, distance=.7cm] (2,0);

\node at (3,0){$\cdots$}  ;

\draw[fill] (4,0) circle (2pt);
\draw (4,0) --node[below]{$n$} (4,0);
\draw (4,0)  to[in=50,out=140,loop, distance=.7cm] (4,0);

\draw(0,0) -- (1,0);
\draw(1,0) -- (2,0);
\draw(2,0) -- (2.7,0);
\draw(3.3,0) -- (4,0);

\end{tikzpicture}$$
\end{definition} 

Homotopies are defined using the product graph $X \times I_n^\ell$.

\begin{definition}\label{D:prod} \cite{HN2004} For graphs $X$ and $Y$, the (categorical) {\bf product graph} $X \times Y$ is defined by: \begin{itemize}
\item  A vertex is a  pair $(v, w)$ where $v \in V(X)$ and $w \in V(Y)$.
\item An edge is defined by $(v_1, w_1) \con (v_2, w_2) \in E(X \times Y) $ whenever $v_1 \con v_2 \in E(X)$ and $w_1 \con w_2 \in E(Y)$.   \end{itemize}
\end{definition}

\begin{definition}\cite{Docht1} \label{D:htpy} Given $f, g:  X \to Y$, we say that $f$ is {\bf $\times$-homotopic} to $g$, written $f \simeq g$,  if there is a map $\Gamma: X \times I_n^{\ell}  \to Y$ such that $\Gamma | _{X \times \{ 0\} } = f$ and $\Gamma | _{X \times \{ n\} } = g$.         \end{definition}
\begin{definition}\label{D:homotopyequivalence}
    Let $f:X\to Y$ and $g:Y\to X$.  If $gf\simeq \id_X$ and $fg\simeq \id_Y$, we say that $f, g$ are $\times$- homotopy equivalences and that $X, Y$ are $\times$-homotopy equivalent,  denoted $X\simeq Y$.
\end{definition}
There are several definitions of homotopy for graphs in the literature.   In this work, we focus exclusively on $\times$-homotopy for this paper and thus `homotopy' will always refer to $\times$-homotopy.

For finite graphs, we have a concrete interpretation of homotopy equivalence via fold maps.

\begin{definition}\label{D:fold}
   Let  $X$ be a graph with $x,x'\in V(X)$ such that $N(x)\subseteq N(x')$.  Define $X'=X\backslash\{x\}$.   The map $\rho:X\to X'$ such that $\rho(x)=x'$ and $\rho(y)=y$ for all $y\in V(X), y\neq x$ is a  \textbf{fold} map, and the inclusion 
    $X'\to X$ is an \textbf{unfold} map.
\end{definition}

\begin{obs}[\cite{CS1}]\label{O:fold}  Any fold map is a homotopy equivalence.  Moreover, 
  for finite graphs $X, Y$,  $\Phi: X\to Y$ is a homotopy equivalence if and only if it can be written as a composition $\varphi_1\varphi_2\cdots\varphi_n = \Phi$ where each  $\varphi_i$ is a fold or unfold. 
\end{obs}

\subsection{Background on Cops and Robbers}
Next we recall the game of Cops and Robbers, originally defined in \cites{Quillot, Nowakowski} and thoroughly discussed in \cite{BonatoCaR}.

\begin{definition}\label{D:CR}
    \textbf{Cops and Robbers} is a game played on a finite graph with two players, the \textbf{cop} player and the \textbf{robber} player.  The cop player begins by placing a number of cops $c_1, \ldots c_m$ on the vertices of the graph, after which the robber players places a robber $R$ on a vertex.   The two players then take turns, with the cop player going first.  On the cop turn, each cop may either remain in place or move to an adjacent vertex.  The robber on their turn has the same choices.  

    If after any move any cop is on the same vertex as the robber, we say that the cop player has \textbf{captured} the robber.   The objective of the cop player is to \textbf{capture} the robber player, and the goal of the robber player is to \textbf{evade} capture indefinitely. 
    
    If the cop player with $k$ cops has a strategy which  always results in a capture of the robber no matter how the robber moves, we say that the cop player has  a \textbf{winning strategy} for $k$ cops.   Given a graph $X$, we define the \textbf{cop number} of $X$ (denoted $c(X)$) to be the minimum number of cops needed for the cops to have a winning strategy on $X$.
\end{definition}

Numerous papers have been written on this game, and many variants defined.  One variant of particular interest for this work is the Fully Active Cops and Robbers of \cite{FullyActive}.

\begin{definition}\label{D:FACR}
    \textbf{Fully Active Cops and Robbers} is a game  with similar rules to the original,  except that on each player's turn, all pieces under the player's control \textit{must} move to an adjacent vertex.  We define the \textbf{fully active cop number} of $X$ (denoted $c_A(X)$) to be the minimum number of cops needed for the cops to have a winning strategy on $X$ with this variant.
\end{definition}

In the classic Cops and Robbers, there is no difference between a looped and unlooped vertex in terms of possible moves.  However, in the fully active variant, a piece on an unlooped vertex must move to a different adjacent vertex, whereas a piece on a looped vertex may traverse the loop and stay on the same vertex, as the vertex is adjacent to itself via the loop.

\section{The basics of Sneaky-Active Cops and Robbers}\label{S:basics}

We begin by introducing our variant of the classic Cops and Robbers game:  \textbf{Sneaky-Active Cops and Robbers}.

\begin{definition}\label{D:SACR}
    \textbf{Sneaky-Active Cops and Robbers}
    is a game played on a finite graph, with two players, the \textbf{cop} player and the \textbf{robber} player.  The cop player begins by placing  cops $c_1, \ldots c_m$ on the vertices of the graph, after which the robber players places a robber $R$ on a vertex.   Then in each turn, starting with the cop player, each player moves the pieces under their control to an adjacent vertex.  The objective of the cop player is to \textbf{capture} the robber player by moving onto the vertex occupied by the robber, and the goal of the robber player is to \textbf{evade} capture indefinitely.
    
    Specifically, the game differs from the classic game in the following way:
    \begin{itemize}
        \item 
   Each piece \textit{must} move to an adjacent vertex at each turn, so a cop or robber may only retain its position if the vertex it occupies has a loop, making it adjacent to itself.   

 \item      A cop captures the robber ONLY if during the cop turn, a cop moves onto the vertex occupied by the robber.   If the robber moves onto a vertex occupied by a cop, this is NOT a capture.   \end{itemize}
\end{definition}

Because every piece must move at every turn, as in the fully active variant of \cite{FullyActive}, we call this variant active.   And 
since the robber may move onto a vertex occupied by a cop without being captured, we imagine that the robber is able to `sneak' by the cops, and call this variant sneaky.

 Given a graph $X$, we let $c_{SA}(X)$ denote the minimum number of cops needed for the cops to have a winning strategy in Sneaky-Active Cops and Robbers.  Thus there is a strategy that allows the cop player to always capture the robber with $k$ cops, but if there are only $k-1$ cops the robber can evade capture indefinitely.

\begin{example}\label{E:SAgame}
    Consider $P_5$:
    \[
        \begin{tikzpicture}
            \draw (1,0) -- (5,0);
            \foreach \x in {1,...,5}{
            \draw[fill, ] (\x, 0)node[below]{\tiny $\x$}  circle (2pt);            };
        \end{tikzpicture}
    \]
    Suppose there were one cop $c$.  Wherever the cop player places $c$, the robber player places the robber $R$ on the same vertex.  When the cop moves, the robber follows them and moves onto the same vertex.  Thus the robber has a strategy to evade capture indefinitely.

    However, suppose there were two cops $c_1, c_2$ placed on vertices  $1,2$ respectively.  If the robber starts on 1 or 3, $c_2$ may capture them on the cop turn.  If the robber starts on $2$, $c_1$ may capture them on the cop turn.  If the robber starts on $4$, $c_1$ has odd distance from the robber, and by moving  to the right on each turn, they will eventually capture the robber.  Similarly if the robber starts on $5$, $c_2$ may capture the robber by moving right each turn. Thus $c_{SA}(P_5)=2$.

\end{example}

The discussion in Example \ref{E:SAgame} leads to the following observation.

\begin{obs}\label{O:cn1} 
    If $X$ has no loops, then $c_{SA}(X)\geq 2$,  since if there is only one cop the robber can evade capture  indefinitely by starting and remaining on the same vertex as the cop as in Example \ref{E:SAgame}.
\end{obs}

The result that motivated our interest in this version of the game is that the cop number $c_{SA}(X)$ is a $\times$-homotopy invariant.

\begin{theorem}\label{T:invariant}
    If  $X, Y$ are $\times$-homotopy equivalent,then  $c_{SA}(X) = c_{SA}(Y)$.
\end{theorem}
\begin{proof}
    By Observation \ref{O:fold}, graphs $X, Y$ are $\times$-homotopy equivalent if and only if one may be transformed into the other via a sequence of folds and unfolds.  Thus it suffices to show that $c_{SA}$ is invariant under a fold map.
    
    Let $X$ be a graph and $X'= X\backslash\{x\} =  \rho(X)$ where $\rho$ is a fold map which takes $x\mapsto x'$.    We first suppose that if $k$ cops can win on $X$ and show that $k$ cops may win on $X'$ as well.  For any starting vertex for the robber $R$ on $X'$, place a shadow robber on any vertex in the fiber $\varphi^{-1}(R)$ in $X$.   Then for any move the robber makes in $X'$, we can lift this to a move of the shadow robber in $X$ since $\varphi$ is surjective on neighborhoods (if there is a choice, pick one arbitrarily).  Then the cops in  $X$ have a winning strategy by assumption which they deploy, responding to the moves of shadow robber in $X$,   and the cops in $X'$  move to  $\varphi$ of the cop positions in $X$.    The shadow cops eventually capture the shadow robber on $X$, and when they do so, they cops capture the robber in $X'$ as well.
    
    Now suppose that $k$ cops can win on $X'$.  Then $k$ cops can win in $X$ according to the following strategy: the cops are placed and move only in the subgraph $X'$ of $X$, and follow their winning strategy in $X'$ if the robber is not on vertex $x$.   If the robber moves onto the vertex $x$, the cops proceed as if the robber moved onto vertex $x'$ instead.   In the next turn, any move by the robber from $x'$ is also a move that could have been made from vertex $x$, and so the game proceeds.
Eventually the cops create a winning capture condition in $X'$.   If the robber is on the vertex $x$ but the cops have moved onto the vertex $x'$ in $X'$, then in the next move the robber is forced onto a vertex which is a neighbor of the vertex $x$, and thus the cop currently on vertex $x$ makes the capture in the next turn.

\end{proof}

\begin{example}\label{E:HomotopyInvariant}
   There is a fold map  $\rho:C_4\to P_3$ defined by $4\mapsto 2, i\mapsto i$ for $i\leq 3$:  
    
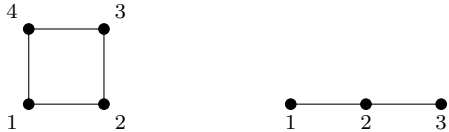
\[
        \begin{tikzpicture}
            \draw (0,0) -- (1,0) -- (1,1) -- (0,1) -- (0,0);
            \draw[fill, ] (0, 0)node[below left]{\tiny $1$}  circle (2pt);
            \draw[fill, ] (1, 0)node[below right]{\tiny $2$}  circle (2pt);
            \draw[fill, ] (1, 1)node[above right]{\tiny $3$}  circle (2pt);
            \draw[fill, ] (0, 1)node[above left]{\tiny $4$}  circle (2pt);
        \end{tikzpicture}
\phantom{wwwwww}
   \begin{tikzpicture}
            \draw (0,0) -- (2,0);
            \draw[fill, ] (0, 0)node[below ]{\tiny $1$}  circle (2pt);
            \draw[fill, ] (1, 0)node[below ]{\tiny $2$}  circle (2pt);
            \draw[fill, ] (2, 0)node[below]{\tiny $3$}  circle (2pt);
        \end{tikzpicture}\]

  In each graph, starting two cops on vertices 1 and 2 results in an immediate cop win, since any starting position for the robber is connected to a cop.    Thus $c_{SA}(C_4) = c_{SA}(P_3) = 2$. 
\end{example}

Note that the classic cops and robbers game, as well as the fully active variant, are \textit{not} homotopy invariant.
\begin{example}\label{E:notinvariant}
     Consider $P_3, C_4$ which are homotopy equivalent as seen in Example \ref{E:HomotopyInvariant}.  In both the classic and fully-active game, the cops have a winning strategy with one cop on $P_3$, but not on $C_4$.  
\end{example}

In the case of reflexive graphs, Sneaky-Active Cops and Robbers is an equivalent game to the classic Cops and Robbers. 

\begin{proposition}\label{P:reflexiveclassic}
  If $X$ is a reflexive graph, then  $c_{SA}(X)=c(X)$. 
 
\end{proposition}
\begin{proof} 

If the graph $X$ is reflexive, then in Sneaky-Active Cops and Robbers the players always have the option to remain on the same vertex by traversing the loop on that vertex, and so the allowed moves are identical to the allowed moves for the classic Cops and Robbers game.    Furthermore, the robbers achieve a capture condition in the classic game either by moving onto the robber or by having the robber move onto them.  The second case is not a capture in the sneaky-active variant, but if the robber moves onto a vertex occupied by a cop, the cop can capture them in the next turn by traversing the loop.  Thus the capture conditions are also equivalent in a reflexive graph.   Therefore on a reflexive graph the games are identical and  $c_{SA}(X)=c(X)$.

\end{proof}

For reflexive graphs, the folding operation corresponds to the dismantling of graphs described in \cite{BonatoCaR}.   Thus the homotopy-invariance of the sneaky-active cop number of Theorem \ref{T:invariant} is equivalent to the invariance of the cop number under dismantling in the classic game {\cite{BonatoCaR} Theorem 2.3} for reflexive graphs.

\begin{example}\label{E:reflexiveequivalent}
Consider the looped path graph $I^\ell_4$:
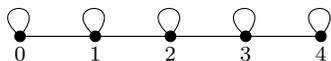
\[
\begin{tikzpicture}
    \foreach \x in {0,...,4}{
    \draw[fill, ] (\x, 0)node[below ]{\tiny $\x$}  circle (2pt);
    \draw (\x,0)  to[in=50,out=140,loop, distance=.7cm] (\x,0);
    };
    \draw (0,0) -- (4,0);

\end{tikzpicture}
\]

 Since  $c(I^\ell_4)=1, c_{SA}(I^\ell_4)=1$  as well by Proposition \ref{P:reflexiveclassic}.

\end{example}

We can use the similarities between our  sneaky-active variant and the classic 
and fully active games to establish some bounds for the sneaky-active cop number relative to the cop number for these other games.

\begin{lemma}\label{L:bounds}
    For any graph $X$,  $c_{{A}}(X)\leq c_{SA}(X)$ and $c(X)-1\leq  c_{SA}(X)\leq 2c(X)$.  
\end{lemma}

    \begin{proof}
    We begin by comparing our game to the fully active variant:  the moves available are exactly the same, but the capture condition in the sneaky-active game is more stringent than the fully active variant, and so  $c_{A}(X)\leq c_{SA}(X)$. 
    This also gives us a comparison to the classic game,  since $c(X)-1\leq c_{A}(X)$ by \cite{FullyActive}, Theorem 3.1.   So   $c(X)-1\leq c_{SA}(X)$ as well.  
    
    Lastly,     if $k$ cops have a winning strategy in the classic game, then $2k$ cops have a winning strategy in the sneaky-active game, since we may place the cops in pairs:  if  $c_1, \ldots, c_k$ are the cop positions in the classic game, then we place another $k$ cops on  $c'_1, \ldots, c'_m$ for $c'_i$  adjacent to $c_i$.   Have the cops on $c_i$ play according to the classic strategy, and have their partners $c'_i$ move with them to remain adjacent.    If the strategy  in the classic game requires the cop to remain stationary on a vertex, then in the sneaky-active game we switch the pair of cops $c_i, c'_i$.    Eventually the cops achieve a capture in the classic game.  If this capture in the classic game is made by having the robber move onto a cop vertex $c_i$, then the partner cop on $c'_i$ will achieve capture in the sneaky-active game in the following turn.   Therefore  $2k$ cops is sufficient in the sneaky-active game and $c_{SA}(X)\leq 2c(X)$.   \end{proof}

We conclude this section with some general observations about  Sneaky-Active Cops and Robbers.   Firstly, for the sneaky-active variant there are some differences in behavior between bipartite and non-bipartite graphs.

\begin{obs}\label{O:notbipartite}
   On a connected non-bipartite graph $X$, if $k$ cops $c_1, \ldots c_k$ have a winning strategy where  $c_i$ starts on vertex $w_i$, then they also have one where $c_i$ starts on vertex $v_i$, so  the start positions of the cops is irrelevant:   since $X$ is not bipartite, there is always an even length walk from $v_i$ to $w_i$.  By appending walks of the form $w_ixw_ix\cdots xw_i$, we may  these walks  have the same length for all $i$, say $\ell$.  Thus, after $\ell$ turns, we can move the cops from $v_i$ to  $w_i$.   Since they have a winning strategy from this position regardless of where the robber starts or has moved to, they may capture the robber.

   This does not hold if $X$ is bipartite.  If a cop $C$ starts on a vertex that is in the same partite set as the robber $R$ chooses to start on, then this cop can never capture the robber, since $d(C, R)$ will always be odd after a cop move.   Thus as soon as the robber chooses a starting vertex, we may disregard any cops that are on vertices of the same partite set. 

   Thus the partite set that the cops start on matters. Within this partite set the starting position does not matter, since the cops can move from one initial configuration to another  in an even number of moves as in the non-bipartite case,  at which point the robber is still in their partite set.

\end{obs}

\begin{lemma}\label{L:bipartite}
    If $X$ is connected and bipartite, then in the winning strategy for $c_{SA}(X)$ cops, the  cops must start with $c_{SA}(X) /2$ cops in each partite set.  
\end{lemma}
\begin{proof}
    Suppose that there is a winning strategy for  $k$ cops placed in one partite set provided that the robber starts on a vertex in the other partite set.   Then $k$ cops placed in the opposite partite set will also have a winning strategy for when the robber chooses the other partite set as a starting point, since after the first move, both robber and cops have swapped partite sets and the cops can play their original winning strategy.  

    Thus if $k$ is the minimal number of cops needed for a winning strategy in one partite set, then $c_{SA}(X) = 2k$.

\end{proof}

Finally, we address non-connectivity of graphs.

\begin{obs}\label{O:notconnected}
    If $X$ is disconnected with $m$ components $X_1, \ldots X_m$,  once a robber is placed on component $X_i$ then only the cops placed on component $X_i$ can capture.    Thus  $c_{SA}(X)=\sum_{i=1}^m c_{SA}(X_i)$.
\end{obs}

\section{Cop numbers for particular graph families}\label{S:families}

We now compute the sneaky-active cop numbers for a variety of graph families, starting with cycles.

\begin{proposition}\label{P:cnoddcycle}
    If a graph is an odd cycle $ C_{2n+1}$, then $c_{SA}(C_{2n+1})=2$.
\end{proposition}
\begin{proof}
    By Observation \ref{O:cn1}, $c_{SA}(C_{2n+1})\geq 2$.  To show that two cops have a winning strategy, suppose we start the cops on vertices $1$ and $2$.  Consider the paths from each cop away from each other to the initial robber position;  these paths are either both even or both odd.  If these lengths are both odd, the cops can capture the robber by moving along these paths towards the robber.  After each cop turn the distance from the robber is even, and so the robber cannot sneak by either cop by moving onto their position and remains stuck between the two and eventually captured.   If the lengths of the paths are both even, the cops switch places on their first turn, making the distance both odd after the robber's first move, and then enact the previously described strategy.

\end{proof}

We  now consider even cycles. We have already seen in Example \ref{E:HomotopyInvariant} that $c_{SA}(C_4)=2$ so we consider larger even cycles.

\begin{proposition}\label{P:cnevencycle}
   If a graph is an even cycle $C_{2n}$ with $n \geq 3$, then   $c_{SA}(C_{2n})=4$. 
\end{proposition}
\begin{proof}
     Since $C_{2n}$ is bipartite,  know by By Lemma \ref{L:bipartite} that the winning cop strategy requires half the cops to be placed in each partite set.   If we have only two cops and place one in each partite set, the robber may evade the relevant cop indefinitely by moving around the circle away from her.   Thus $c_{SA}(C_{2n})$ is at least $4$.   Then by  \cite{BonatoCaR}, the classic cop number $c(C_{2n})=2$  and so by Lemma \ref{L:bounds}, $c_{SA}(C_{2n})\leq 2\cdot 2=4$.  Thus $c_{SA}(C_{2n})=4$.   
\end{proof}

\begin{example}\label{E:evencycle}

     Consider the bipartite graph $C_8$.    With $4$ cops, we can place them adjacent to each other as shown:  
    
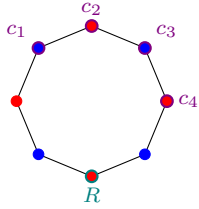
\[
        \begin{tikzpicture}
            \draw ({sin(0)},{cos(0)})\foreach \x in {1,...,8}{--({sin(45*\x)}, {cos(45*\x)})};
            \foreach \x in {1,...,4}{
            \draw[fill, red] ({sin(90*\x)}, {cos(90*\x)})  circle (2pt);
            \draw[fill, blue] ({sin(90*\x+45)}, {cos(90*\x+45)})  circle (2pt);
            };
            \draw[ thick, violet] ({sin(-45)},{cos(-45)}) node[above left]{\tiny $c_1$} circle (2.2pt);
            \draw[ thick, violet] ({sin(0)},{cos(0)}) node[above]{\tiny $c_2$} circle (2.2pt);
            \draw[ thick, violet] ({sin(45)},{cos(45)}) node[above right]{\tiny $c_3$} circle (2.2pt);
             \draw[ thick, violet] ({sin(90)},{cos(90)}) node[ right]{\tiny $c_4$} circle (2.2pt);
            \draw[ thick, teal] ({sin(45*4)},{cos(45*4)}) node[below]{\tiny $R$} circle (2.2pt);
        \end{tikzpicture}
    \]
   No matter where the robber starts, there will be two cops which have odd distance from the robber's initial position who can move around the cycle in opposite directions and  capture the robber.

\end{example}

Next we consider complete graphs.  

\begin{proposition}\label{P:cncomplete}
For a complete graph $K_n$ with $n \geq 2$, we have   $c_{SA}(K_n)=2$. 
\end{proposition}
\begin{proof}
    By Observation \ref{O:cn1}, $c_{SA}(K_n)\geq 2$.   If we start two cops on distinct vertices,  then any starting robber position is a neighbor of at least one of the cops and the robber is immediately captured. 
\end{proof}

\begin{corollary}\label{C:tree}
    If a graph is a tree $T$, then $c_{SA}(T)=2$.
\end{corollary}
\begin{proof}
    Any tree may be folded down to $K_2$, starting with its leaves, and so  $T\simeq K_2$.    By Theorem \ref{T:invariant} we have that $c_{SA}(T)=c_{SA}(K_2)=2$.
\end{proof}

Lastly we look at  Kneser graphs.  

\begin{definition}[\cite{Bondy}]\label{D:Kneser}
    A \textbf{Kneser graph} denoted $K(n,m)$ has vertex set $V(K(n,m))=\{S\subseteq [n]: |S|=m\}$, i.e.\ subsets of $\{1,\ldots, n\}$ of size $m$.  We have that $S_1\sim S_2$ when $S_1\cap S_2$.
\end{definition}

Recall that the \textit{Petersen graph} is the Kneser graph $K(5,2)$.

\begin{theorem}\label{T:cnpetersen}
    The sneaky-active cop number of the Peterson graph  $c_{SA}(K(5,2))= 3$.
\end{theorem}
\begin{proof}
    We first show that $c_{SA}(K(5,2))\leq 3$. Place cops on $\{a,b\}, \{c,d\}, \{a,d\}$ where $a,b,c,d,e\in [5]$ are distinct.  To avoid immediate capture, the  robber must start on either $\{a,c\}, \{a,d\}$, or $\{b,d\}$.

    Suppose the robber begins on $\{a,c\}$.  Then the $\{a,b\}, \{c,d\}$ cops switch and $\{a,d\}$ moves to $\{c,e\}$.  The robber's neighbors are $\{b,e\}, \{b,d\}$, and $\{d,e\}$ all of whom are neighbors of one of the cops and they are caught.  Similarly if the robber starts on $\{b,d\}$, they are caught.

    If the robber starts on $\{a,d\}$, move the $\{a,d\}$ cop to $\{c,e\}$ and have the $\{a,b\}, \{c,d\}$ cops switch.  The robber then has to move to $\{b,c\}$.  Then, have the $\{a,b\}, \{c,e\}$ cops switch and have the $\{c,d\}$ cop move to $\{a,b\}$.  The robber has to move to $\{a,e\}$.  Finally switch $\{a,b\}, \{c,e\}$ with one of the $\{a,b\}$ cops and move the other to $\{d,e\}$.  The robber's neighbors are $\{b,c\}, \{b,d\}$, and $\{c,d\}$, all of whom are neighbors of a cop and thus the robber is captured.

    We next show that $c_{SA}(K(5,2)) > 2$. For any cop on vertex $c$ and robber on vertex $R$, we note that $|N(c)\cap N(R)|\leq 2$ with equality holding if and only if $c=R$.  So suppose there are two cops on vertices $c_1, c_2$.  Note that on any robber turn we have that that $c_1, c_2\neq R$, there can be at most 2 neighbors of $R$ that are also neighbors of either cop.  Since $d(R)=3$, there is a position the robber may move to in order to evade capture in the following cop turn. 
\end{proof}

\begin{proposition}
    For a Kneser graph $K(n,2)$ with $n \geq 6$, we have $c_{SA}(K(n,2))=3$. 
\end{proposition}
\begin{proof} 

    We first show that 2 cops $c_1, c_2$ do not suffice.  There are two possible initial cop positions up to symmetry:  either the cops are adjacent or they are not.  In the first case the cops are adjacent, that is they are on vertices $\{a,b\}, \{c,d\}$, and the robber may place themself on a vertex non-adjacent to both.  If the cops are not adjacent and are on $\{a,b\}, \{b,c\}$, the robber may likewise place themself non-adjacent to both.  Now after each robber move, we suppose that neither cop has captured the robber, so $|c_i\cap R|\leq 1$.  If the cops are adjacent, then there is a vertex non-adjacent to both cops but also adjacent to the robber that the robber may move to in order to evade capture.  If the cops are not adjacent, then they occupy vertices $\{a,b\}, \{b,c\}$, then either the $b\in R$, and the robber may move to $\{a,c\}$ or $b\not\in R$ and the robber may move to $\{b,d\}$ evading capture.

    On the other hand, if 3 cops start on $\{1,2\}, \{3,4\}, \{5,6\}$, then no matter where the robber starts and moves, they will be adjacent to at least one cop and be captured.
\end{proof}

We may summarize these findings in Figure \ref{fig:tableofcopnumbers}. 

\begin{figure}[h]
    \[
    \begin{array}{|c||c|c|c|c|c|c|}
    \hline
    X& C_{2k+1} & C_{2k} & \text{Trees} & K_n &  K(n,2), n\geq 5 \\
    \hline
    c_{SA}(X)& 2&4 & 2 & 2  & 3 \\
    \hline
    \end{array}
    \]
    \caption{A table of cop numbers.}
    \label{fig:tableofcopnumbers}
\end{figure}

\section{The Categorical Product}\label{S:CP}

In this section, we consider the  categorical product of graphs of Definition \ref{D:prod}, and  relate the sneaky-active cop number of a product graph $X \times Y$ to the cop numbers of the graphs $X$ and $Y$.   We have already seen results in which bipartiteness plays a factor in the behavior of the sneaky-active cop number in Section \ref{S:basics}, and this section will show this phenomena continuing;  our results will depend on the bipartiteness of the various graphs involved.

To set the stage for the arguments which follow, recall that the vertices of $X \times Y$ are defined by the product of sets $V(X) \times V(Y)$, and edges are defined between $(x_1, y_1)$ and $(x_2, y_2)$ whenever there are edges $x_1 \sim x_2$ in $E(X)$ and $y_1 \sim y_2$ in $E(Y)$.   Thus any position of a cop or robber in $X \times Y$ is given by a position in $X$ and a position in $Y$, and a move in $X\times Y$ is given by a move in $X$ and a move in $Y$.   Hence a strategy in the product $X \times Y$ is given by two simultaneous strategies in the component graphs $X, Y$.

With this in mind, we begin by considering our first case, where none of the graphs involved are bipartite.  

\begin{theorem}\label{T:notbipartite}
    Suppose $X, Y$ are both connected and not bipartite.   Then $c_{SA}(X \times Y) = c_{SA}(X) + c_{SA}(Y)-1 $.

\end{theorem}

\begin{proof}
    Suppose we have $c_{SA}(X) + c_{SA}(Y) -1$ cops.  We place our cops arbitrarily.  By Observation \ref{O:notbipartite}, we know that we have winning strategies in both $X$ and $Y$ regardless of initial placement.    
    
    Phase 1:   we start with all cops playing a winning strategy in both $X$ and $Y$. 
  This will lead to a cop win in either $X$ or $Y$, meaning that we have a cop that moves onto a position with the same $x$ or $y$ coordinate as the robber. 
  If this cop 'captures' in $X$, we move her into Team $Y$;  if she 'captures' in $Y$, she is moved to Team $X$.  After this reassignment, we begin again:  all Team $Y$ cops will now  follow the robber in $X$,  while the Team $Y$ and unassigned cops together play  a winning strategy  $Y$, and all Team $X$ cops will follow the robber in $Y$ while Team $X$ and unassigned cops together play  a winning strategy  $X$.  

  If a cop from Team $X$ makes a capture in $X$, then the robber has lost, since that cop also shares the robber's $Y$ coordinate;  similarly if a cop from Team $Y$ makes the capture in $Y$, the robber loses.
    We continue until either the robber is captured in $X \times Y$ or until we do not have enough cops left in one of the groups to have a winning strategy.      This second case will happen  when  we drop to  $c_{SA}(X) -1$ cops in the combined Team $X$ and unassigned groups,  or $c_{SA}(Y) -1$ cops in the combined Team $Y$ and unassigned groups.   At this point   we move to Phase 2.  

    Phase 2: Assume without loss of generality that we have only $c_{SA}(X) -1$ cops  remaining in the combined Team $X$ and unassigned groups, and so the remaining cops have been assigned to Team $Y$, giving us $c_{SA}(Y)$ cops in Team $Y$.  These cops can now capture the robber by playing a winning strategy in $Y$ while shadowing the robber's position in $X$.     
    
    Thus the cop player  can always win with $c_{SA}(X) + c_{SA}(Y)-1$ cops.  To show that the robber can win if there are fewer cops, suppose there are  at most $c_{SA}(X) -1 + c_{SA}(Y) -1$ cops.    The robber  has a strategy to avoid $c_{SA}(X)-1$ cops in $X$, and  a strategy to avoid $c_{SA}(Y)-1$ cops in $Y$.    Thus the robber can avoid all cops by arbitrarily dividing them into  $c_{SA}(X) -1$ `avoid in $X$' cops and $c_{SA}(Y)-1$ `avoid in $Y$' cops and playing the avoidance strategy in each graph,  hence avoiding all cops in $X \times Y$. 
\end{proof}

\begin{example}  \label{E:nobipartite} Consider  cycle graphs $C_i, C_j$ where $i, j$ are both odd, $i,j\geq 3$.    Then $c_{SA} (X) = c_{SA} (Y) = 2$ and so $c_{SA}(X \times Y) = 3$.   We illustrate the strategy of Theorem \ref{T:notbipartite} in this example.  
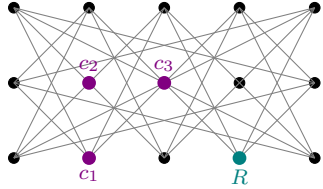
\[\begin{tikzpicture}
        \foreach \x in {0,...,4}{
        \foreach \y in {0,...,2}{
        \draw[fill] ({\x}, {\y})  circle (2pt);
        };
        };
         \foreach \x in {0,...,3}{
        \draw[gray] ({\x},0)--({\x+1},1);
        \draw[gray] ({\x},0)--({\x+1},2);
        \draw[gray] ({\x},1)--({\x+1},0);
        \draw[gray] ({\x},1)--({\x+1},2);
        \draw[gray] ({\x},2)--({\x+1},0);
        \draw[gray] ({\x},2)--({\x+1},1);
        };
        \draw[gray] (0,0) edge[ ] (4,1);
        \draw[gray] (0,0) edge[ ] (4,2);
        \draw[gray] (0,2) edge[ ] (4,1);
        \draw[gray] (0,1) edge[ ] (4,0);
        \draw[gray] (0,1) edge[ ] (4,2);
        \draw[gray] (0,2) edge[ ] (4,0);

        \draw[ thick, violet, fill] (1,0) node[below ]{\tiny $c_1$} circle (2.2pt);
        \draw[ thick, violet, fill] (1,1) node[above ]{\tiny $c_2$} circle (2.2pt);
        \draw[ thick, violet, fill] (2,1) node[above ]{\tiny $c_3$} circle (2.2pt);

        \draw[ thick, teal, fill] (3,0) node[below ]{\tiny $R$} circle (2.2pt);
    \end{tikzpicture}\]
     At the beginning, all three cops play a winning strategy in both $X$ and $Y$.  If Cop 1 captures the robber in $X$, she gets reassigned to  Team $Y$ and begins shadowing the robber in $X$.  Cop 2 and Cop 3 play their winning strategy in $X$, while all three cops play as a team in  $Y$.   Suppose next that Cop 2 captures the robber in $Y$, and is assigned to Team $X$, while only Cop 3 remains unassigned.    Now we have Cops 1 and 3 playing to win in $Y$, Cops 2 and 3 playing to win in $X$, while Cop 1 shadows the robber in $X$ and Cop 2 shadows him in $Y$.   When the next cop makes a capture in either $X$ or $Y$, we move into Phase 2:  if  Cop 3 captures the robber in $X$, we now have Cop 1 and Cop 3 on the same $X$ coordinate as the robber and these two cops can play to win in $Y$ while continuing to shadow in $X$.  When one of them achieves a capture in $Y$, it results in a win in $X \times Y$.  

   If there are only two cops,   the robber can choose  one to avoid in $X$  while avoiding the other in $Y$.     Since one cop is insufficient to capture the robber in a cycle graph, the robber can evade these two cops in $X \times Y$ indefinitely.
\end{example}

Since Sneaky-Active Cops and Robbers is equivalent  to the classic Cops and Robbers for reflexive graphs (Proposition \ref{P:reflexiveclassic}), reflexive graphs are not bipartite and  the categorical and strong products are equivalent for reflexive graphs, we recover the following result for the classic game.

\begin{corollary}
    For the classic cops and robbers game, $c(X\boxtimes Y)=c(X)+c(Y)-1$.
\end{corollary}

When one of the graphs $X$ or $Y$ is bipartite, we must split the cops into two groups, one for each partite set.   Otherwise the strategy unfolds in a similar way to the non-bipartite case.

\begin{theorem}\label{T:onebipartite}
    Suppose $X$ is not bipartite and $Y$ is bipartite.     Then $c_{SA}(X \times Y) = 2c_{SA}(X) + c_{SA}(Y)-2 $.

\end{theorem}

\begin{proof}
    Since $Y$ is bipartite, we know by Lemma \ref{L:bipartite} that in order to win in $Y$, the cop player needs to start with  $c_{SA}(Y)/2$ cops placed in each partite set.   Now $X \times Y$ is  bipartite, so  with $2c_{SA}(X) + c_{SA}(Y)-2$  the cop player can choose a starting placement with  $c_{SA}(X) + c_{SA}(Y)/2 -1 $ cops in each partite set.    Once the robber chooses a start in one partite set  then  half of the cops are irrelevant by Observation \ref{O:notbipartite}, and we have $c_{SA}(X) + c_{SA}(Y)/2-1$ cops who are placed in the correct partite set and can potentially capture the robber.   These cops have a winning strategy similar to the non-bipartite case.   
    
   Phase 1 is to have the relevant cops play winning strategies in $X$ and $Y$, moving them into Team $X$ and Team $Y$ as they reach a position with a matching $y$ or $x$ coordinate, and then having Team $X$ cops shadow in $Y$ while continuing to play a winning strategy in $X$, while Team $Y$ cops shadow in $X$ and play in $Y$.   This continues until the cops capture the robber or have too few cops to have a winning strategy, which only happens when we have reached  $c_{SA}(X)$ cops in Team $X$ or $c_{SA}(Y)/2$ cops in Team $Y$.   At this stage we enter Phase 2 and either Team $X$ cops can win in $X$ while shadowing in $Y$, or $c_{SA}(Y)/2$ cops of Team $Y$ (who are positioned in the relevant partite set) can win in $Y$ while shadowing in $X$.     

   To show that $2c_{SA}(X) + c_{SA}(Y)-3 $ are not sufficient,  observe that since  $X \times Y$ is bipartite, once the robber chooses a start position, then all cops
who are currently in  the same partite set as the robber's initial position are irrelevant.    Have the robber start on whichever partite set has more cops, hence fewer cops who can potentially make a capture after moving.     Then there are at most $c_{SA}(X) -1 + c_{SA}(Y)/2 -1$ cops in play, and the robber can divide them into ones to be avoided in $X$ and ones to be avoided in $Y$ as in the non-bipartite case. 
\end{proof}

\begin{example}\label{E:crossonebipartite}
 Consider  cycle graphs $C_i, C_j$ for $i$  is odd, $i \geq 3$  and $j$ even, $j\geq 6$.     Then  $c_{SA}(C_i \times C_j) =  6 $.  In order for $6$ cops to win, we place $3$ of them on each partite set of $V(C_i \times C_j)$.   Once the robber places himself, we have $3$ cops who are positioned to make a potential capture. 
  As in Example \ref{E:nobipartite}, the cops play in both $C_i$ and $C_j$ until one captures in $C_i$;   then this cop shadows in $C_i$ while the other cops continue to play, and all three play in $C_j$.   Eventually we have two cops who have captured the robber in one of the cycles, say  $C_i$, who then  play to win in $C_j$ while remaining shadowing the robber in $C_i$.  
  
  With at most $5$ cops, the robber places themselves in the partite set which contains only two cops, and proceeds to evade one in $C_i$ and the other in $C_j$.  \end{example}
Lastly, we consider the following case.  
\begin{theorem}\label{T:bothbipartite}
    Suppose $X, Y$ are both bipartite.      Then $c_{SA}(X \times Y) = 2c_{SA}(X) + 2c_{SA}(Y)-4 $.

\end{theorem}

\begin{proof} In this case, $X \times Y$ is a disconnected graph with two components, each of which is bipartite.   Thus we have $4$ partite sets to consider, two for each component of the graph.     We place $c_{SA}(X) + c_{SA}(Y) -1$ cops on each of these partite sets.   Once the robber has chosen a start position, only one out of these four groups of cops  remains in play, and those cops have a winning strategy with the two phase strategy outlined in the proofs of Theorems \ref{T:notbipartite}, \ref{T:onebipartite}. 

   To see that $2c_{SA}(X) + 2c_{SA}(Y)-5 $ cops are insufficient, we note that the robber can always choose an initial placement in the component and partite set which has at most $c_{SA}(X)/2-1 + c_{SA}(Y)/2-1$  cops, and avoid $c_{SA}(X)/2-1$ of them in $X$ and the others in $Y$.  
\end{proof}

  \begin{example}  \label{E:crossbothbipartite}  If $i, j$ are both even and  $i, j\geq 6$, then $c_{SA}(C_i \times C_j ) = 12$.   The graph $C_i \times C_j$ consists of two disconnected bipartite components, and we place 3 cops on each partite set of each component.  With fewer cops, the robber can select the partite set of the component that contains at most $2$ cops and avoid one in each coordinate.

\end{example}

\section{The Box Product}\label{S:BP} 

In this section, we turn to the box product. We establish bounds on the sneaky-active cop number for box products of graphs, and compute some examples.

\begin{definition}\label{D:bprod}  For graphs $X$ and $Y$, the  {\bf box product graph} $X \square Y$ is defined by: \begin{itemize}
\item  A vertex is a  pair $(v, w)$ where $v \in V(X)$ and $w \in V(Y)$.
\item An edge is defined by $(v_1, w_1) \con (v_2, w_2) \in E(X \times Y) $ whenever either \begin{itemize}
    \item $v_1 = v_2 \in V(X)$ and $w_1 \con w_2 \in E(Y)$ or 
    \item $v_1 \con v_2 \in E(X)$ and $w_1 = w_2 \in V(Y)$
\end{itemize} \end{itemize}
\end{definition}
Hence  a move in the box product $X \square Y$ will correspond to a move in $X$ {\bf or} a move in $Y$, but not both; and at every step, each cop or robber will choose which direction to move in, either $X$ or $Y$.   

We begin with the case where both graphs are not bipartite.

\begin{theorem}\label{T:notbipartiteboxbound}
    If $X$ and $Y$ are \textit{not} bipartite graphs then $c_{SA}(X \square Y) \leq c_{SA}(X) + c_{SA}(Y)$. 
    
\end{theorem}

\begin{proof}

Suppose we have $c_{SA}(X)+ c_{SA}(Y)$ cops.     The cops have a winning strategy as follows. 

Phase 1: During this phase, the cops always move in the same direction as the robber does.   So if the robber makes a move in $X$, then all cops move in $X$, and similarly for $Y$.  To begin, all cops are unassigned and  play together as a team following winning strategies for $X$ and $Y$ respectively.    When a cop 'captures' the robber in $X$ by moving onto a vertex with the same $X$ coordinate as the robber, the cop is moved onto Team $Y$.  From then on, they shadow the robber in $X$ and continue playing with the Team $X$ and unassigned cops in $Y$.  Similarly cops who 'capture' the robber in $Y$ are moved onto Team $X$. 

The cops continue this strategy, reassigning cops as they 'capture' in one coordinate, until we reach $c_{SA}(X)$ cops in Team $X$  or $c_{SA}(Y)$ cops assigned to Team $Y$.  At this point, the strategy changes.

Phase 2:  The Team $X$ cops continue playing a winning strategy in $X$ and shadowing in $Y$, and a capture in $X$ results in a cop win in $X \square Y$.   Hence the robber can only make a finite number of moves in $X$ from this point on.   The Team $Y$ cops shadow the robber in $X$ and move randomly in $Y$ when the robber moves in $Y$.   The unassigned cops, however, move ONLY in $X$ during this phase: no matter what the robber does, they move towards the robber's $X$ position in $X$ until they match it and join Team $Y$.

Phase 3:  Once all remaining unassigned cops have joined Team $Y$, we now have  $c_{SA}(X)$ cops in Team $X$ whose $Y$-coordinate matches the robber, and $c_{SA}(Y)$ cops in Team $Y$ whose $X$-coordinate matches the robber.  At this point the cops revert to the strategy of the first phase,  where all cops move in the same coordinate as the robber does and  Team $X$ plays to win in $X$ and shadows the robber in $Y$, and Team $Y$ plays to win in $Y$ and shadows in $X$.   Since each team has a winning strategy in their respective factor, the robber is captured in a finite number of turns.  

\end{proof}

When both graphs are bipartite, then the strategy of Theorem \ref{T:notbipartiteboxbound} still works provided we are careful about our initial placement of our cops.  

\begin{theorem}\label{T:bipartiteboxbound}
    If $X$ and $Y$ are \textit{both} bipartite graphs then $c_{SA}(X \square Y) \leq c_{SA}(X) + c_{SA}(Y)$. 
    
\end{theorem}

\begin{proof}
    In this case, we know by Lemma \ref{L:bipartite} that the winning strategy in both $X$ and $Y$ consist of placing the cops as two equal groups of size $c_{SA}(X)/2$ and $c_{SA}(Y)/2$ respectively in each partite set of  $V(X)$ and $V(Y)$.   Now  $X \square Y$ is also bipartite, so with $c_{SA}(X) + c_{SA}(Y)$ cops  we can place $c_{SA}(X)/2 + c_{SA}(Y)/2$  cops in each partite set of $X \square Y$.     Once the robber chooses their initial position only the cops in the opposite partite set are relevant.  We enact the strategy of the proof of Theorem \ref{T:notbipartiteboxbound} using the relevant cops.  Note that as the cops join Team $X$ and $Y$, they will automatically wind up in the relevant partite set of $V(X)$ or $V(Y)$ when they match the other coordinate.    
\end{proof}

Interestingly, when $X$ is bipartite but $Y$ is not, we can improve the bound.     In this case $X\square Y$ is not bipartite. We may still mimic the arguments made in the proofs of Theorems \ref{T:notbipartiteboxbound}, \ref{T:bipartiteboxbound}, using only the $c_{SA}(X)/2$ to capture the robber in $X$ and using the non-bipartiteness of $Y$ to maneuver them into the right partite set in $X$.  

\begin{theorem}\label{T:halfbipartiteboxbound}
    If $X$ is bipartite and $Y$ is not, then $c_{SA}(X \square Y) \leq {c_{SA}(X)}/{2} + c_{SA}(Y)$. 
    
\end{theorem}

\begin{proof}  
 Suppose we have ${c_{SA}(X)}/{2}+ c_{SA}(Y)$ cops. We will follow the same basic strategy as in the earlier cases, but we need to ensure that the cops who are assigned to Team $X$ wind up in the correct partite set of $V(X)$ to play their strategy there once their $Y$ coordinates match the robber.

  Phase 0: We again start with all cops unassigned.  This time, on the first turn once the robber has placed, we look at the $X$-coordinate of each cop.   If the coordinate is in the opposite (the `correct') partite set of $V(X)$ as that of the robber, the cop moves in $X$;  if they are in the `wrong' partite set, the cop makes their initial move in $Y$.   Thus at the end of the turn, all cops have an $X$ coordinate matching that of the robber.

Phase 1:   All cops always move in the same direction as the robber does, either in $X$ or $Y$.   When a cop matches the $X$-coordinate of the robber, they move to Team $Y$ and begin shadowing in $X$, and when they match the $Y$-coordinate they move to Team $X$ and shadow in $Y$.  The Team $X$ and unassigned groups play together in $X$ when the robber moves in $X$, and the Team $Y$ and unassigned groups play together in $Y$ when the robber moves in $Y$.      Note that when a cop 'wins' in $Y$ and is moved into Team $X$, their $X$ coordinate will be in the correct partite set of $V(X)$ because of the Phase 0 move.

  This phase continues until there are either ${c_{SA}(X)}/{2}$ cops in Team $X$ or $c_{SA}(Y)$ cops in Team $Y$;  up to this point, there will always be enough cops playing together in $X$ and in $Y$ to force a win,  so we can keep moving cops from the unassigned group into either Team $X$ or $Y$.

Phase 2:  We need to consider two cases here.  
\begin{itemize}
    \item[Case 1:] There are $c_{SA}(X)/2$ cops in Team $X$ and thus we have $c_{SA}(Y)$ other cops either  in Team $Y$ or unassigned. Because of the first move of Phase $0$ and the fact that all cop moved in $X$ exactly when the robber did during Phase 1, these  $c_{SA}(X)/2$ cops all have $X$ coordinates that are in the relevant partite set of $V(X)$.  The Team $X$ cops continue playing a winning strategy in $X$ and shadowing in $Y$, and a capture in $X$ results in a cop win in $X \square Y$,  and so the robber can only make a finite number of moves in $X$ from this point on without being captured, and is forced to eventually move in $Y$.     The Team $Y$ cops shadow the robber in $X$ and move {arbitrarily} in $Y$ when the robber moves in $Y$.   The unassigned cops move \textbf{only} in $X$ during this phase: no matter what the robber does, they move towards the robber's $X$ position in $X$ until they match it and join Team $Y$.

    \item[Case 2:]  There are $c_{SA}(Y)$ cops in Team $Y$ and the rest either in Team $X$, with matching $Y$ coordinates and in the relevant partite set of $V(X)$, or are unassigned. In this scenario, the Team $Y$ cops have a winning strategy and so the robber is forced to move in $X$ after a finite number of moves.   We  move the unassigned cops into Team $X$ when they have a matching $Y$ coordinate as the robber and are in the relevant partite set of $V(X)$.   These cops therefore move to reduce the distance from the robber in $Y$, regardless of how the robber moves, and after finite moves, match  the robber's $Y$ coordinate.   If they are also in the correct $V(X)$  partite set, they are now assigned to be partite set of Team $X$.   If they are not, they traverse an odd length circuit  in $Y$ and return to the same coordinate, and then continue to move towards the robber's new $Y$ position.

\end{itemize}

Phase 3:  In either case,  we now have $c_{SA}(X)/2$ cops in Team $X$ with matching $Y$-coordinate to the robber and in the correct partite set of $V(X)$, and $c_{SA}(Y)$ cops in Team $Y$ with matching $X$-coordinate to the robber.    So as in the proofs of Theorems \ref{T:notbipartiteboxbound}, \ref{T:bipartiteboxbound}, when the robber moves in $X$, Team $X$ plays their winning strategy in $X$ and $Y$ shadows the robber and vice versa, ensuring a capture.  
\end{proof}

The next examples are graphs in which the bounds of Theorems \ref{T:notbipartiteboxbound}, \ref{T:bipartiteboxbound} and \ref{T:halfbipartiteboxbound} are realized.  

\begin{example}\label{E:tightboxbound1}
    For two non-bipartite graphs as in Theorem \ref{T:notbipartiteboxbound}, consider the graph $K_3^\ell$, a looped $K_3$ where each vertex is adjacent to each vertex including itself;  a single cop can move to any vertex, including the one it is on, and so  $c_{SA}(K_3^\ell)=1$. Then $K_3^\ell \square K_3^\ell$ is depicted below: 

    
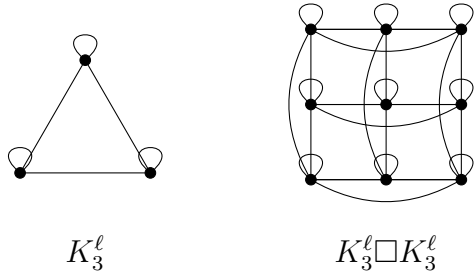
\[
    \begin{tikzpicture}
    \node at (-2,0){\begin{tikzpicture}
        \draw ({sin(0)},{cos(0)})\foreach \x in {1,...,3}{--({sin(360/3*\x)}, {cos(360/3*\x)})};
        \foreach \x in {1,...,3}{            
            \draw[fill] ({sin(360/3*\x)}, {cos(360/3*\x)})  circle (2pt);
        };
        \draw ( 0.866025403784, -1/2 ) to[in=-220,out=40,loop, distance=.7cm]  ( 0.866025403784,-1/2 ); 
        \draw ( -0.866025403784, -1/2 ) to[in=-220,out=40,loop, distance=.7cm]  ( -0.866025403784,-1/2 );
        \draw ( 0,1 ) to[in=-220,out=40,loop, distance=.7cm]  ( 0,1 );

    \end{tikzpicture}};

    \node at (-2,-2){ $K_3^\ell$};
    
    \node at (2,0){\begin{tikzpicture}
        \foreach \x in {0,...,2}{
        \draw (\x, 0)--(\x, 2);
        \draw (\x,0)edge[bend left](\x,2);
        };
        
        \draw (0,0)--(2,0);
        \draw (0,2)--(2,2);
        \draw (0,1)--(2,1);
        
        \foreach \y in {0,...,2}{
        \draw (0, \y)--(2, \y);
        \draw (0,\y)edge[bend right](2,\y);
        };

        \foreach \x in {0,...,2}{
        \foreach \y in {0,...,2}{
            \draw ( \x,\y ) to[in=-220,out=40,loop, distance=.7cm]  ( \x,\y );
        };
        }'
        
        \draw[fill, black] (0,0) circle (2pt);
        \draw[fill, black] (1,0) circle (2pt);
        \draw[fill, black] (2,0) circle (2pt);
        \draw[fill, black] (0,1) circle (2pt);
        \draw[fill, black] (1,1) circle (2pt);
        \draw[fill, black] (2,1) circle (2pt);
        \draw[fill, black] (0,2) circle (2pt);
        \draw[fill, black] (1,2) circle (2pt);
        \draw[fill, black] (2,2) circle (2pt);

         \end{tikzpicture}};
         
        \node at (2,-2){ $K_3^\ell\square K_3^\ell$}; 
    \end{tikzpicture}
    \]

    By Proposition \ref{P:reflexiveclassic}, playing the sneaky-active game on this reflexive graph is equivalent to playing the classic game on $K_3\square K_3$. One cop is insufficient here, as the robber can always move away from the cop.   Two cops can win, since cops placed at  $(1,1), (2,2)$ can move to any vertex on their first move.  Thus $c_{SA}(K_3^\ell\square K_3^\ell)=2=c_{SA}(K_3^\ell)+c_{SA}(K_3^\ell)$.
\end{example}

\begin{example}\label{E:tightboxbound2}
    For the bipartite case of Theorem \ref{T:bipartiteboxbound}, consider $K_2 \square C_4$ and recall that $c_{SA}(K_2)=c_{SA}(C_4)=2$.  Since $C_4\cong K_2\square K_2$, the box product is isomorphic the hypercube $\displaystyle \square_{i=1}^3 K_2$:

     
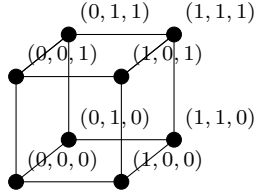
\[\begin{tikzpicture}[scale=1.4]

\foreach \x in {0,...,1}{
\foreach \y in {0,...,1}{

\draw (\x,\y) --node[above right]{\tiny $(\x, 0, \y)$} (\x, \y);

\draw[fill] (\x,\y) circle (2pt);

}
};
\foreach \x in {0,...,1}{
\draw (\x, 0)--(\x, 1);
};

\draw  (0, 0)--(1, 0);
\draw (0, 1)--(1, 1);

\def\sx{0.5}
\def\sy{0.4}

\foreach \x in {0,...,1}{
\foreach \y in {0,...,1}{

\draw (\x+\sx,\y+\sy) --node[above right]{\tiny $(\x, 1, \y)$} (\x+\sx, \y+\sy);

\draw[fill] (\x+\sx,\y+\sy) circle (2pt);
\draw (\x, \y) -- (\x+\sx, \y+\sy);
}
\draw  (\x, 1) -- (\x+\sx, 1+\sy);
};

\foreach \x in {0,...,1}{
\draw (\x+\sx, 0+\sy)--(\x+\sx, 1+\sy);
};

\draw (0 + \sx,  0+\sy)--(1+\sx, 0+\sy);
\draw (0 + \sx,  1+\sy)--(1+\sx, 1+\sy);

\node at (1,-1){$K_2\square C_4\cong K_2^3$};

 \end{tikzpicture}\]

Placing $4$ cops on $(0,0,0), (0,1,0), (1,0,0), (1,1,0)$ ensures capture of the robber on the first cop move  no matter what the robber initial position is. If we have only $3$ cops,  one of the partite sets contains at most $1$ cop and the robber can evade them by placing themselves on top of this cop and then following them around.    Thus $c_{SA}(\square_{i=1}^3 K_2)=4=c_{SA}(K_2)+c_{SA}(C_4)$.

    \end{example}

\begin{example}\label{E:tightboxbound3}

For the case where one graph is bipartite as in Theorem \ref{T:halfbipartiteboxbound}, consider $T\square C_4$ where $T$ represents the terminal graph with one vertex and one loop.   Then  $c_{SA}(T)=1, c_{SA}(C_4)=2$, and $T\square C_4\cong C_4^\ell$, the reflexive 4-cycle, and playing the sneaky-active game on  $T\square C_4$ is equivalent to playing the classic game on $C_4$.  So $c_{SA}(T\square C_4)=c(C_4)=2=c_{SA}(T) + c_{SA}(C_4)/2$.

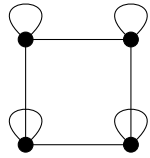
\[
\begin{tikzpicture}[scale=1.4]
    \foreach \x in {0,...,1}{
\foreach \y in {0,...,1}{

\draw[fill] (\x,\y) circle (2pt);
\draw ( \x,\y ) to[in=-220,out=40,loop, distance=.7cm]  ( \x,\y );
}
};
\draw (0,0)--(0,1)--(1,1)--(1,0)--(0,0);
\node at (1/2,-0.8){$T\square C_4\cong C_4^\ell$};

\end{tikzpicture}
\]

\end{example}

We next use the bound of Theorem \ref{T:bipartiteboxbound} and the homotopy invariance of Theorem \ref{T:invariant}  to compute the sneaky-active cop number for the box product of trees.

We begin with a result on homotopy equivalence.  
\begin{theorem}\label{T:boxtree}
    If $T_1, T_2$ are trees  then $T_1 \square T_2$ is homotopy equivalent to $K_2$. 

    \end{theorem}

\begin{proof} 
 Our argument rests on the following result:  if  $x$ is a leaf of of $T_1$ and we define $\hat{T}_1$ to be the tree $T_1 \backslash \{x\}$ created by deleting the vertex $x$, then  $T_1 \square T_2 \simeq \hat{T}_1 \square T_2$. We do this by defining a sequence of folds that sequentially delete all vertices $(x, y)$ in $T_1 \square T_2$ for $y \in V(T_2)$.   To define these folds, root tree $T_2$ and assume it has height $m$.   We start with vertices $y_m$ of level $m$ in $T_2$.    
 Since both $x$ and $y_m$ are leaves of their respective trees, there is a fold map that takes  $(x, y_m) $ to $(\hat{x}, \hat{y}_m)$ where $\hat{x},\hat{y}_m$ are the parents of vertices $x, y_m$ respectively.
 Once all vertices  $(x, y_m)$ for $y_m$ of level $m$ are deleted, we then fold vertices $(x, y_{m-1})$ for all level  $m-1$  vertices $y_{m-1}$ of $T_2$. The only remaining neighbors of vertices $(x, y_{m-1})$ are $(\hat{x}, y_{m-1})$ and $(x, \hat{y}_{m-1})$, and so $(x, y_{m-1}) $  can now be folded to $(\hat{x}, \hat{y}_{m-1})$.  We continue to delete vertices $(x, y_j)$ via the fold $(x, y_j)\mapsto (\hat{x}, \hat{y}_j)$, working through the levels $m\geq j\geq 1$ in $T_2$ until all vertices of the form $(x,y)$ have been deleted, leaving us with $\hat{T}_1 \square T_2$.

This result forms the basis for an inductive argument that $T_1 \square T_2 \simeq K_2$, inducting on the number of vertices of the tree $T_1$:  if $|V(T_1)| = 2$ then $\hat{T}_1 \square T_2 = T_2$ and any tree is homotopy equivalent to $K_2$, and if $|V(T_1)| = k+1$ then when we delete a leaf,  $|V(\hat{T}_1)| = k$ and so we get $T_1 \square T_2 \simeq \hat{T}_1 \square T_2 \simeq K_2$ by our inductive hypothesis.   
\end{proof}

\begin{corollary}\label{C:boxtree}
    If $T_1, T_2$ are trees   then $c_{SA}(T_1 \square T_2)=2$.
\end{corollary}

We extend this result to an arbitrary box product of trees.

\begin{theorem}\label{T:boxmanytrees}
    If $X$ is a box product of $2n$ trees, then $c_{SA}(X)=2n$.  Similarly, if $X$ is a box product of   $2n-1$  trees, then $c_{SA}(X)=2n$.
\end{theorem}
\begin{proof} Suppose that  $\displaystyle X=\square_{i=1}^{2n}T_i, n\geq 1$   where $T_i$ is a tree.  
    By Theorem \ref{T:boxtree}, $c_{SA}(\square_{i=1}^{2}T_i)=2$ and so by Theorem \ref{T:bipartiteboxbound} and induction we have $c_{SA}(\square_{i=1}^{2n}T_i)=c_{SA}((\square_{i=1}^{2n-2}T_i)\square (\square_{i=1}^{2}T_i))\leq 2n$. 
    
    Now  suppose we have only $2n-1$ cops.  Since $X$ is bipartite, we may place the robber so that at most $n-1$ cops in the relevant partite set to make a capture.  The robber is either on top of a cop, in which case they share at most $2$ neighbors with that cop, or away from the cop, giving at most $1$ shared neighbors;  so in total there can be at most  $2n-2$ positions which the robber cannot move to.       Since the degree of each vertex  $\delta(X)\geq 2n$,   there is always a neighbor of the robber they may move to in order to avoid capture in the following turn.

The argument for the odd case is similar, as  Theorem \ref{T:bipartiteboxbound} and induction give 
      $c_{SA}(\square_{i=1}^{2n-1}T_i)=c_{SA}((\square_{i=1}^{2n-3}T_i)\square (\square_{i=1}^{2}T_i))\leq 2n$ and the degree of each vertex is $2n-1$, which is still larger than $2n-2$. 
\end{proof}

\begin{example}\label{E:hypercube}
    Recall that Example \ref{E:tightboxbound2} showed that $c_{SA}(\square_{i=1}^3K_2)=4$, which is exactly Theorem \ref{T:boxmanytrees} when $n=2$.
\end{example}

We can utilize some results from \cite{FullyActive} and \cite{CRProduct} to exhibit some bounds about the box product of cycles for the sneaky-active variant. 

We begin with a list of results we will use.  
\begin{lemma}[\cite{FullyActive} Lemma 4.5]\label{L:bipartitecopnumber}
          Let $X$ be a bipartite graph. Let $k = c(X)$, and consider the fully active game on $X$ with $k$ cops. If, after the initial placement by both players, all cops and the robber occupy the same partite set of $X$, then the cops can ensure capture of the robber.
\end{lemma}

We note here that in the case where all cops occupy the same partite set of $X$ as the robber after a cop move, then there is no distinction between the sneaky-active game and the active game since the robber is not adjacent to any cops and thus may not ``sneak" onto any cop.  Thus Lemma \ref{L:bipartitecopnumber} applies to the sneaky-active game as well.

We also utilize active cop numbers of the box product of cycles from \cite{FullyActive} and \cite{CRProduct}.

\begin{theorem}[\cite{FullyActive} Theorem 4.8]\label{T:fullyactivecycles}
    Let $X=C_{n_1}\square C_{n_2}\square \cdots \square C_{n_k}$.  If any of the $n_i$ are odd then $c_a(X)\leq k+1$.
          
\end{theorem}

\begin{theorem}[\cite{CRProduct} Theorem 2.6]\label{T:productcycles}
    Let $X = \displaystyle \square_{i=1}^n C_i$ where $C_i$ is a cycle of length at least 4.  Then $c(X)=n+1$.
          
\end{theorem}

We use these results to find bounds for the cop number of box products of cycles in our sneaky-active case.

\begin{theorem}\label{T:boxoddcycle}   {Consider a box product of cycles $\square_{i=1}^kC_{n_i}$.}  
If any of the $n_i$ are odd, then $k\leq c_{SA}(\square_{i=1}^kC_{n_i})\leq k+1$.  
\end{theorem}
\begin{proof}
    We first show that $k\leq c_{SA}(\square_{i=1}^kC_{n_i})$. Without loss of generality, let the position of the robber be $r=(0,0,\ldots,0)$. If the robber $r$ and cop $c_i$ share a vertex as a neighbor, then $d(r, c_i)=2$, and there is either a coordinate where $c_i=(0,\ldots, \pm 2, 0)$ or a pair of coordinates where $c_i=(0, \ldots, \pm 1,\ldots, \pm 1, \ldots,0 )$.  Since each neighbor of $r$ has the form $(0,\ldots, \pm1, \ldots, 0)$, we have that $|N(r, c_i)|\leq 2$.  Since $\square_{i=1}^kC_{n_i}$ is a regular graph with degree $2k$, at least $k$ cops are needed to cover all neighbors of $r$, else the robber may move to a neighbor and avoid capture in the following turn.

    For the upper bound, observe that $c_{SA}(\square_{i=1}^kC_{n_i})\leq c_A(\square_{i=1}^kC_{n_i})\leq k+1$ by Lemma \ref{L:bounds} and Theorem \ref{T:fullyactivecycles}.
\end{proof}

\begin{theorem} \label{T:boxevencycle}   
 Consider a box product of even cycles $\square_{i=1}^kC_{2n_i}$.  We have that $c_{SA}(\square_{i=1}^kC_{2n_i})\in \{2k, 2k+2\}$.
\end{theorem}

\begin{proof}
    Since $\square_{i=1}^kC_{2n_i}$ is bipartite, it suffices to show that the number of cops needed in each partite set is  $k$ or $k+1$.  The degree argument from the proof of Theorem \ref{T:boxoddcycle} shows that we need at least $k$ cops per partite set.  Then by Lemma \ref{L:bipartitecopnumber} and Theorem \ref{T:productcycles}, we have that  $k+1$ cops in each partite set is sufficient to capture the robber.  Thus $c_{SA}(\square_{i=1}^kC_{2n_i})=2k$ or $2k+2$.
\end{proof}

We now consider box products of complete graphs.  

\begin{lemma}\label{L:completemin}
     Let $K_n, K_m$ be complete graphs.   Then $c_{SA}(K_m\square K_n) \leq \min(m,n) $.
\end{lemma}
\begin{proof}
    Without loss of generality, suppose $m\leq n$, then place cops initially on $(1,1), (2,1), \ldots, (m,1).$  For any initial robber position, there is a cop adjacent to this position.
\end{proof}

\begin{proposition}\label{P:boxK3leq}
     Let  $K_m, K_n$ be complete graphs with $m,n\geq 3$.   Then  $c_{SA}(K_m\square K_n) $ $ \geq 3$.
\end{proposition}
\begin{proof}
 
    We begin with the case where $n,m=3$. Suppose there are two cops $c_1, c_2$. We consider the three possible initial configurations for two cops relative to each other:  they are either on the same vertex (I), connected (II) or disconnected (III).   In each case, we place the robber as shown in the diagram below;  in the grid shown, any vertices in the same row or the same column are connected in $K_3 \square K_3$, and any others are not.     From here,  no matter how the cops move, the robber can move back to the positions shown:  if the cops move into configurations (A)-(F) (shown up to symmetry), the robber moves back into a desired configuration (I)-(III) as indicated by the arrows in the diagram below.

    
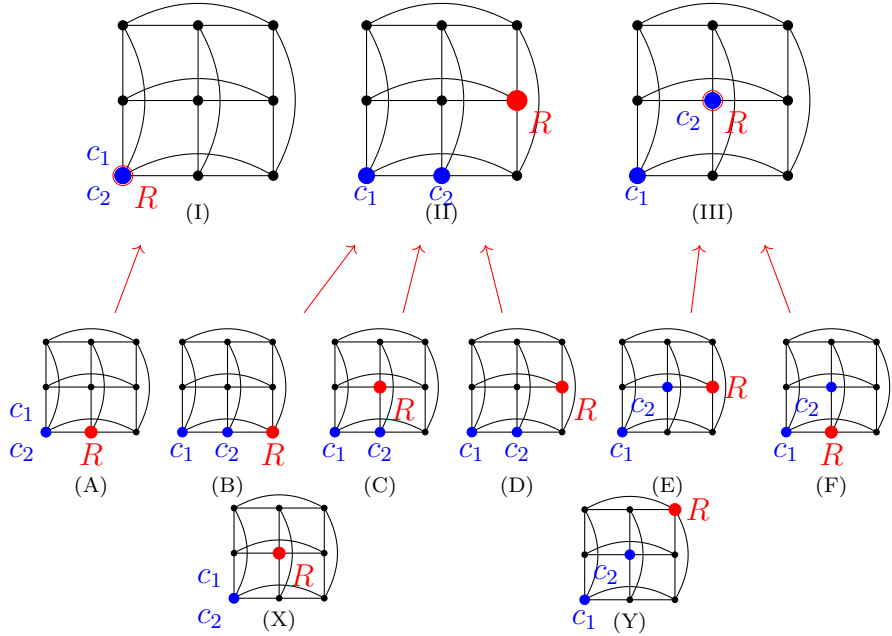
\[
    \begin{tikzpicture}
        \node (R1) at (-5,0){
        \begin{tikzpicture}[scale=0.6]
         \foreach \x in {0,...,2}{
             \draw (0,\x)--(2,\x);
             \draw (0,\x) edge[bend left] (2,\x);
             \draw (\x,0)--(\x,2);
             \draw (\x,0) edge[bend right] (\x,2);
             \foreach \y in {0,...,2}{
                \draw[fill] (\x,\y) circle (1.75pt);
             }
         }
          \draw[fill, blue] (0,0) node[above left]{$c_1$}  circle (3pt);
          \draw[fill, blue] (0,0) node[below left]{$c_2$}  circle (3pt);
          \draw[red,fill ] (1,0) node[below ]{$R$}  circle (3.75pt);
          \node at (1,-1.2){\tiny (A)};
         \end{tikzpicture}
        };

        \node (R2) at (-3,0){
        \begin{tikzpicture}[scale=0.6]
         \foreach \x in {0,...,2}{
             \draw (0,\x)--(2,\x);
             \draw (0,\x) edge[bend left] (2,\x);
             \draw (\x,0)--(\x,2);
             \draw (\x,0) edge[bend right] (\x,2);
             \foreach \y in {0,...,2}{
                \draw[fill] (\x,\y) circle (1.75pt);
             }
         }
          \draw[fill, blue] (0,0) node[below]{$c_1$}  circle (3pt);
          \draw[fill, blue] (1,0) node[below ]{$c_2$}  circle (3pt);
          \draw[red, fill] (2,0) node[below ]{$R$}  circle (3.75pt);
          \node at (1,-1.2){\tiny (B)};
         \end{tikzpicture}
        };
        
        \node (R3) at (-1,0){
        \begin{tikzpicture}[scale=0.6]
         \foreach \x in {0,...,2}{
             \draw (0,\x)--(2,\x);
             \draw (0,\x) edge[bend left] (2,\x);
             \draw (\x,0)--(\x,2);
             \draw (\x,0) edge[bend right] (\x,2);
             \foreach \y in {0,...,2}{
                \draw[fill] (\x,\y) circle (1.75pt);
             }
         }
          \draw[fill, blue] (0,0) node[below]{$c_1$}  circle (3pt);
          \draw[fill, blue] (1,0) node[below ]{$c_2$}  circle (3pt);
          \draw[red, fill] (1,1) node[below right]{$R$}  circle (3.75pt);
          \node at (1,-1.2){\tiny (C)};
         \end{tikzpicture}
        };

        \node (R4) at (1,0){
        \begin{tikzpicture}[scale=0.6]
         \foreach \x in {0,...,2}{
             \draw (0,\x)--(2,\x);
             \draw (0,\x) edge[bend left] (2,\x);
             \draw (\x,0)--(\x,2);
             \draw (\x,0) edge[bend right] (\x,2);
             \foreach \y in {0,...,2}{
                \draw[fill] (\x,\y) circle (1.75pt);
             }
         }
          \draw[fill, blue] (0,0) node[below]{$c_1$}  circle (3pt);
          \draw[fill, blue] (1,0) node[below ]{$c_2$}  circle (3pt);
          \draw[red, fill] (2,1) node[below right]{$R$}  circle (3.75pt);
          \node at (1,-1.2){\tiny (D)};
         \end{tikzpicture}
        };

        \node (R5) at (3,0){
        \begin{tikzpicture}[scale=0.6]
         \foreach \x in {0,...,2}{
             \draw (0,\x)--(2,\x);
             \draw (0,\x) edge[bend left] (2,\x);
             \draw (\x,0)--(\x,2);
             \draw (\x,0) edge[bend right] (\x,2);
             \foreach \y in {0,...,2}{
                \draw[fill] (\x,\y) circle (1.75pt);
             }
         }
          \draw[fill, blue] (0,0) node[below ]{$c_1$}  circle (3pt);
          \draw[fill, blue] (1,1) node[below left]{$c_2$}  circle (3pt);
          \draw[red, fill] (2,1) node[ right]{$R$}  circle (3.75pt);
          \node at (1,-1.2){\tiny (E)};
         \end{tikzpicture}
        };

        \node (R6) at (5,0){
        \begin{tikzpicture}[scale=0.6]
         \foreach \x in {0,...,2}{
             \draw (0,\x)--(2,\x);
             \draw (0,\x) edge[bend left] (2,\x);
             \draw (\x,0)--(\x,2);
             \draw (\x,0) edge[bend right] (\x,2);
             \foreach \y in {0,...,2}{
                \draw[fill] (\x,\y) circle (1.75pt);
             }
         }
          \draw[fill, blue] (0,0) node[below ]{$c_1$}  circle (3pt);
          \draw[fill, blue] (1,1) node[below left]{$c_2$}  circle (3pt);
          \draw[red, fill] (1,0) node[ below]{$R$}  circle (3.75pt);
          \node at (1,-1.2){\tiny (F)};
         \end{tikzpicture}
        };

        \node (C1) at (-3.5,4){
        \begin{tikzpicture}
         \foreach \x in {0,...,2}{
             \draw (0,\x)--(2,\x);
             \draw (0,\x) edge[bend left] (2,\x);
             \draw (\x,0)--(\x,2);
             \draw (\x,0) edge[bend right] (\x,2);
             \foreach \y in {0,...,2}{
                \draw[fill] (\x,\y) circle (1.75pt);
             }
         }
          \draw[fill, blue] (0,0) node[above left]{$c_1$}  circle (3pt);
          \draw[fill, blue] (0,0) node[below left]{$c_2$}  circle (3pt);
          \draw[red, ] (0,0) node[below right]{$R$}  circle (3.75pt);
          \node at (1,-0.5){\tiny (I)};
         \end{tikzpicture}
        };

        \node (C2) at (0,4){
        \begin{tikzpicture}
         \foreach \x in {0,...,2}{
             \draw (0,\x)--(2,\x);
             \draw (0,\x) edge[bend left] (2,\x);
             \draw (\x,0)--(\x,2);
             \draw (\x,0) edge[bend right] (\x,2);
             \foreach \y in {0,...,2}{
                \draw[fill] (\x,\y) circle (1.75pt);
             }
         }
          \draw[fill, blue] (0,0) node[below]{$c_1$}  circle (3pt);
          \draw[fill, blue] (1,0) node[below ]{$c_2$}  circle (3pt);
          \draw[red, fill] (2,1) node[below right]{$R$}  circle (3.75pt);
          \node at (1,-0.5){\tiny (II)};
         \end{tikzpicture}
        };

        \node (C3) at (3.5,4){
        \begin{tikzpicture}
         \foreach \x in {0,...,2}{
             \draw (0,\x)--(2,\x);
             \draw (0,\x) edge[bend left] (2,\x);
             \draw (\x,0)--(\x,2);
             \draw (\x,0) edge[bend right] (\x,2);
             \foreach \y in {0,...,2}{
                \draw[fill] (\x,\y) circle (1.75pt);
             }
         }
          \draw[fill, blue] (0,0) node[below ]{$c_1$}  circle (3pt);
          \draw[fill, blue] (1,1) node[below left]{$c_2$}  circle (3pt);
          \draw[red, ] (1,1) node[below right]{$R$}  circle (3.75pt);
          \node at (1,-0.5){\tiny (III)};
         \end{tikzpicture}
        };

        \draw[red,->] (R1) edge (C1);
        \draw[red,->] (R2) edge (C2);
        \draw[red,->] (R3) edge (C2);
        \draw[red,->] (R4) edge (C2);
        \draw[red,->] (R5) edge (C3);
        \draw[red,->] (R6) edge (C3);

        \node (B1) at (-2.5,-2){\begin{tikzpicture}[scale=0.6]
         \foreach \x in {0,...,2}{
             \draw (0,\x)--(2,\x);
             \draw (0,\x) edge[bend left] (2,\x);
             \draw (\x,0)--(\x,2);
             \draw (\x,0) edge[bend right] (\x,2);
             \foreach \y in {0,...,2}{
                \draw[fill] (\x,\y) circle (1.75pt);
             }
         }
          \draw[fill, blue] (0,0) node[above left]{$c_1$}  circle (3pt);
          \draw[fill, blue] (0,0) node[below left]{$c_2$}  circle (3pt);
          \draw[red,fill ] (1,1) node[below right]{$R$}  circle (3.75pt);
          \node at (1,-0.5){\tiny (X)};
         \end{tikzpicture}};

         \node (B2) at (2.5,-2){\begin{tikzpicture}[scale=0.6]
         \foreach \x in {0,...,2}{
             \draw (0,\x)--(2,\x);
             \draw (0,\x) edge[bend left] (2,\x);
             \draw (\x,0)--(\x,2);
             \draw (\x,0) edge[bend right] (\x,2);
             \foreach \y in {0,...,2}{
                \draw[fill] (\x,\y) circle (1.75pt);
             }
         }
          \draw[fill, blue] (0,0) node[below ]{$c_1$}  circle (3pt);
          \draw[fill, blue] (1,1) node[below left]{$c_2$}  circle (3pt);
          \draw[red, fill] (2,2) node[ right]{$R$}  circle (3.75pt);
          \node at (1,-0.5){\tiny (Y)};
         \end{tikzpicture}};

    \end{tikzpicture}
    \]

    The cops cannot move from configurations (I)-(III) to either (X) or (Y).  For  configurations (I) and (III), the cop $c_2$ must share a coordinate with the robber after moving.   From configuration (II), any movement where the cops end sharing 0 or 2 coordinates forces at least one of them to share a coordinate with the robber as well, and thus (II) may not move to (X) or (Y).

    For $n>3$ or $m>3$, note that the cops and the robber lie on some $K_3\square K_3$ subgraph.  We then consider possible configurations on this subgraph.  The argument here now follows as above, except that now all configurations can now be reached from our initial configurations (I)-(V) but that we may use the extra vertices to move back from any of these configurations back into (I)-(V), possibly on a different  $K_3\square K_3$ subgraph.

     
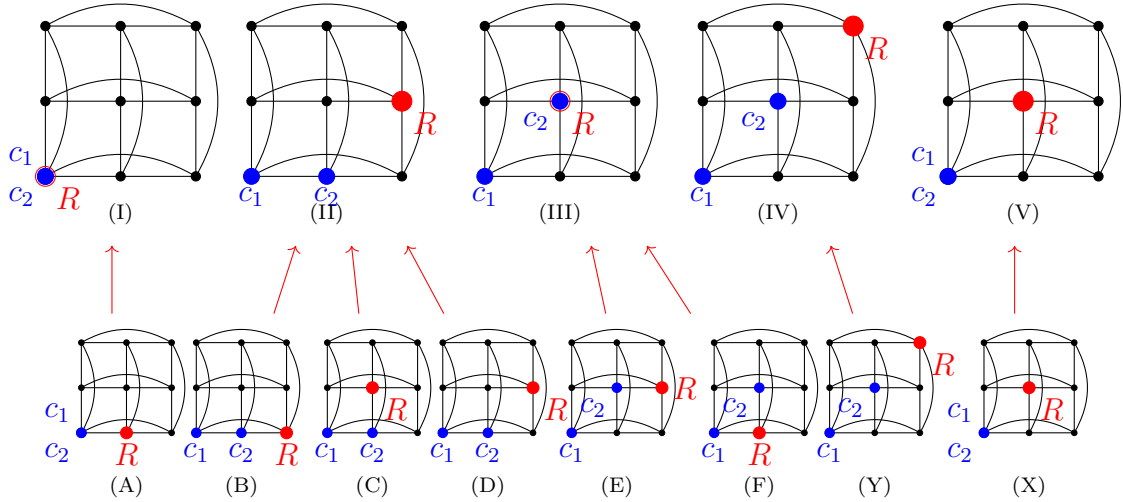
\[
    \begin{tikzpicture}
        \node (R1) at (-6/7*7,0){
        \begin{tikzpicture}[scale=0.6]
         \foreach \x in {0,...,2}{
             \draw (0,\x)--(2,\x);
             \draw (0,\x) edge[bend left] (2,\x);
             \draw (\x,0)--(\x,2);
             \draw (\x,0) edge[bend right] (\x,2);
             \foreach \y in {0,...,2}{
                \draw[fill] (\x,\y) circle (1.75pt);
             }
         }
          \draw[fill, blue] (0,0) node[above left]{$c_1$}  circle (3pt);
          \draw[fill, blue] (0,0) node[below left]{$c_2$}  circle (3pt);
          \draw[red,fill ] (1,0) node[below ]{$R$}  circle (3.75pt);
          \node at (1,-1.2){\tiny (A)};
         \end{tikzpicture}
        };

        \node (R2) at (-6/7*5,0){
        \begin{tikzpicture}[scale=0.6]
         \foreach \x in {0,...,2}{
             \draw (0,\x)--(2,\x);
             \draw (0,\x) edge[bend left] (2,\x);
             \draw (\x,0)--(\x,2);
             \draw (\x,0) edge[bend right] (\x,2);
             \foreach \y in {0,...,2}{
                \draw[fill] (\x,\y) circle (1.75pt);
             }
         }
          \draw[fill, blue] (0,0) node[below]{$c_1$}  circle (3pt);
          \draw[fill, blue] (1,0) node[below ]{$c_2$}  circle (3pt);
          \draw[red, fill] (2,0) node[below ]{$R$}  circle (3.75pt);
          \node at (1,-1.2){\tiny (B)};
         \end{tikzpicture}
        };
        
        \node (R3) at (-6/7*3,0){
        \begin{tikzpicture}[scale=0.6]
         \foreach \x in {0,...,2}{
             \draw (0,\x)--(2,\x);
             \draw (0,\x) edge[bend left] (2,\x);
             \draw (\x,0)--(\x,2);
             \draw (\x,0) edge[bend right] (\x,2);
             \foreach \y in {0,...,2}{
                \draw[fill] (\x,\y) circle (1.75pt);
             }
         }
          \draw[fill, blue] (0,0) node[below]{$c_1$}  circle (3pt);
          \draw[fill, blue] (1,0) node[below ]{$c_2$}  circle (3pt);
          \draw[red, fill] (1,1) node[below right]{$R$}  circle (3.75pt);
          \node at (1,-1.2){\tiny (C)};
         \end{tikzpicture}
        };

        \node (R4) at (-6/7*1,0){
        \begin{tikzpicture}[scale=0.6]
         \foreach \x in {0,...,2}{
             \draw (0,\x)--(2,\x);
             \draw (0,\x) edge[bend left] (2,\x);
             \draw (\x,0)--(\x,2);
             \draw (\x,0) edge[bend right] (\x,2);
             \foreach \y in {0,...,2}{
                \draw[fill] (\x,\y) circle (1.75pt);
             }
         }
          \draw[fill, blue] (0,0) node[below]{$c_1$}  circle (3pt);
          \draw[fill, blue] (1,0) node[below ]{$c_2$}  circle (3pt);
          \draw[red, fill] (2,1) node[below right]{$R$}  circle (3.75pt);
          \node at (1,-1.2){\tiny (D)};
         \end{tikzpicture}
        };

        \node (R5) at (6/7*1,0){
        \begin{tikzpicture}[scale=0.6]
         \foreach \x in {0,...,2}{
             \draw (0,\x)--(2,\x);
             \draw (0,\x) edge[bend left] (2,\x);
             \draw (\x,0)--(\x,2);
             \draw (\x,0) edge[bend right] (\x,2);
             \foreach \y in {0,...,2}{
                \draw[fill] (\x,\y) circle (1.75pt);
             }
         }
          \draw[fill, blue] (0,0) node[below ]{$c_1$}  circle (3pt);
          \draw[fill, blue] (1,1) node[below left]{$c_2$}  circle (3pt);
          \draw[red, fill] (2,1) node[ right]{$R$}  circle (3.75pt);
          \node at (1,-1.2){\tiny (E)};
         \end{tikzpicture}
        };

        \node (R6) at (6/7*3,0){
        \begin{tikzpicture}[scale=0.6]
         \foreach \x in {0,...,2}{
             \draw (0,\x)--(2,\x);
             \draw (0,\x) edge[bend left] (2,\x);
             \draw (\x,0)--(\x,2);
             \draw (\x,0) edge[bend right] (\x,2);
             \foreach \y in {0,...,2}{
                \draw[fill] (\x,\y) circle (1.75pt);
             }
         }
          \draw[fill, blue] (0,0) node[below ]{$c_1$}  circle (3pt);
          \draw[fill, blue] (1,1) node[below left]{$c_2$}  circle (3pt);
          \draw[red, fill] (1,0) node[ below]{$R$}  circle (3.75pt);
          \node at (1,-1.2){\tiny (F)};
         \end{tikzpicture}
        };

        \node (R7) at (6/7*5,0){\begin{tikzpicture}[scale=0.6]
         \foreach \x in {0,...,2}{
             \draw (0,\x)--(2,\x);
             \draw (0,\x) edge[bend left] (2,\x);
             \draw (\x,0)--(\x,2);
             \draw (\x,0) edge[bend right] (\x,2);
             \foreach \y in {0,...,2}{
                \draw[fill] (\x,\y) circle (1.75pt);
             }
         }
          \draw[fill, blue] (0,0) node[below ]{$c_1$}  circle (3pt);
          \draw[fill, blue] (1,1) node[below left]{$c_2$}  circle (3pt);
          \draw[red, fill] (2,2) node[ below right]{$R$}  circle (3.75pt);
          \node at (1,-1.2){\tiny (Y)};
         \end{tikzpicture}};

         \node (R8) at (6/7*1*7,0){\begin{tikzpicture}[scale=0.6]
         \foreach \x in {0,...,2}{
             \draw (0,\x)--(2,\x);
             \draw (0,\x) edge[bend left] (2,\x);
             \draw (\x,0)--(\x,2);
             \draw (\x,0) edge[bend right] (\x,2);
             \foreach \y in {0,...,2}{
                \draw[fill] (\x,\y) circle (1.75pt);
             }
         }
          \draw[fill, blue] (0,0) node[above left]{$c_1$}  circle (3pt);
          \draw[fill, blue] (0,0) node[below left]{$c_2$}  circle (3pt);
          \draw[red,fill ] (1,1) node[below right]{$R$}  circle (3.75pt);
          \node at (1,-1.2){\tiny (X)};
         \end{tikzpicture}};

        \node (C1) at (-6,4){
        \begin{tikzpicture}
         \foreach \x in {0,...,2}{
             \draw (0,\x)--(2,\x);
             \draw (0,\x) edge[bend left] (2,\x);
             \draw (\x,0)--(\x,2);
             \draw (\x,0) edge[bend right] (\x,2);
             \foreach \y in {0,...,2}{
                \draw[fill] (\x,\y) circle (1.75pt);
             }
         }
          \draw[fill, blue] (0,0) node[above left]{$c_1$}  circle (3pt);
          \draw[fill, blue] (0,0) node[below left]{$c_2$}  circle (3pt);
          \draw[red, ] (0,0) node[below right]{$R$}  circle (3.75pt);
          \node at (1,-0.5){\tiny (I)};
         \end{tikzpicture}
        };

        \node (C2) at (-3,4){
        \begin{tikzpicture}
         \foreach \x in {0,...,2}{
             \draw (0,\x)--(2,\x);
             \draw (0,\x) edge[bend left] (2,\x);
             \draw (\x,0)--(\x,2);
             \draw (\x,0) edge[bend right] (\x,2);
             \foreach \y in {0,...,2}{
                \draw[fill] (\x,\y) circle (1.75pt);
             }
         }
          \draw[fill, blue] (0,0) node[below]{$c_1$}  circle (3pt);
          \draw[fill, blue] (1,0) node[below ]{$c_2$}  circle (3pt);
          \draw[red, fill] (2,1) node[below right]{$R$}  circle (3.75pt);
          \node at (1,-0.5){\tiny (II)};
         \end{tikzpicture}
        };

        \node (C3) at (0,4){
        \begin{tikzpicture}
         \foreach \x in {0,...,2}{
             \draw (0,\x)--(2,\x);
             \draw (0,\x) edge[bend left] (2,\x);
             \draw (\x,0)--(\x,2);
             \draw (\x,0) edge[bend right] (\x,2);
             \foreach \y in {0,...,2}{
                \draw[fill] (\x,\y) circle (1.75pt);
             }
         }
          \draw[fill, blue] (0,0) node[below ]{$c_1$}  circle (3pt);
          \draw[fill, blue] (1,1) node[below left]{$c_2$}  circle (3pt);
          \draw[red, ] (1,1) node[below right]{$R$}  circle (3.75pt);
          \node at (1,-0.5){\tiny (III)};
         \end{tikzpicture}
        };

        \node (C4) at (3,4){\begin{tikzpicture}
         \foreach \x in {0,...,2}{
             \draw (0,\x)--(2,\x);
             \draw (0,\x) edge[bend left] (2,\x);
             \draw (\x,0)--(\x,2);
             \draw (\x,0) edge[bend right] (\x,2);
             \foreach \y in {0,...,2}{
                \draw[fill] (\x,\y) circle (1.75pt);
             }
         }
          \draw[fill, blue] (0,0) node[below ]{$c_1$}  circle (3pt);
          \draw[fill, blue] (1,1) node[below left]{$c_2$}  circle (3pt);
          \draw[red, fill] (2,2) node[ below right]{$R$}  circle (3.75pt);
          \node at (1,-0.5){\tiny (IV)};
         \end{tikzpicture}};

         \node (C5) at (6,4){\begin{tikzpicture}
         \foreach \x in {0,...,2}{
             \draw (0,\x)--(2,\x);
             \draw (0,\x) edge[bend left] (2,\x);
             \draw (\x,0)--(\x,2);
             \draw (\x,0) edge[bend right] (\x,2);
             \foreach \y in {0,...,2}{
                \draw[fill] (\x,\y) circle (1.75pt);
             }
         }
          \draw[fill, blue] (0,0) node[above left]{$c_1$}  circle (3pt);
          \draw[fill, blue] (0,0) node[below left]{$c_2$}  circle (3pt);
          \draw[red,fill ] (1,1) node[below right]{$R$}  circle (3.75pt);
          \node at (1,-0.5){\tiny (V)};
         \end{tikzpicture}};

        \draw[red,->] (R1) edge (C1);
        \draw[red,->] (R2) edge (C2);
        \draw[red,->] (R3) edge (C2);
        \draw[red,->] (R4) edge (C2);
        \draw[red,->] (R5) edge (C3);
        \draw[red,->] (R6) edge (C3);
        \draw[red,->] (R7) edge (C4);
        \draw[red,->] (R8) edge (C5);

    \end{tikzpicture}
    \]

\end{proof}

\begin{theorem}\label{T:K3boxgeq}  Let  $K_m, K_n$ be complete graphs with $m,n\geq 3$. 
    Then $c_{SA}(K_m\square K_n)  = 3$. 
\end{theorem}
\begin{proof}
    By Proposition \ref{P:boxK3leq}, it suffices to show that three cops are sufficient.  Without loss of generality, suppose that $3\leq n\leq m$.  If $n=3$, then by Lemma \ref{L:completemin}, three cops are sufficient.

    Suppose that  $n>3$.  To show that three cops are sufficient, place the cops on $c_1=(1,1)$, $c_2=(2,2)$, $c_3=(3,3)$.  The robbers initial placement must either be on one of these cop positions or on $(x,y)$ where $x,y\not\in \{1,2,3\}$ to avoid capture in the next turn.   

     In the first case,  suppose without loss of generality the robber starts on $(3,3)$.  Then  move $c_1$ to $(1,3)$,  $c_2$ to $(1,2)$ and $c_3$ to $(3,2)$.    Then any move made by the robber results in capture on the next turn.  

     In the second case, if the robber starts on $(x,y)$ where $x,y\not\in \{1,2,3\}$, move  $c_1$ to $(x,1)$, $c_2$ to $(2,y)$ and $c_3$ to $(3,y)$   To avoid immediate capture, the robber must move to $(x,1)$.  Then move $c_1$ to $(1,1)$, $c_2$ to $(1,y)$ and $c_3$ to $(x,y)$.  Like in the first case, all valid robber moves lead to  capture in the following cop turn.

\end{proof}

\section{Future Directions} 

While we have some box product results from Section \ref{S:BP}, these bounds could potentially be further tightened.  In particular, when considering Theorem \ref{T:boxoddcycle}, observational evidence suggests that $c_{SA}(\square_{i=1}^k)C_{n_i}=k+1$, but this is not proven in this paper.  Moreover, while we have results for the box products of arbitrarily many trees or cycles, the results for complete graphs is restricted to the product of a pair of graphs and so the sneaky-active cop number for a general box product of complete graphs is still open.  

While this paper focuses on $\times$-homotopy, there are numerous other discrete homotopies defined on graphs, such as $A$-homotopy, equivalence via neighborhood complex \cite{Matsushita}, digital homtopy \cite{DigitalHomotopy} and more.  All of these homotopy theories result in their own notions of homotopy equivalence, and developing Cops and Robber variants which are invariant under these homotopies could shed light on the behavior of these different theories.

% BibTeX users please use one of
%\bibliographystyle{spbasic}      % basic style, author-year citations
\bibliographystyle{spmpsci}      % mathematics and physical sciences
\bibliography{ref}   % name your BibTeX data base

\end{document}